\documentclass[12pt,leqno,amsfonts,amscd]{amsart}
\usepackage{epsfig}
\usepackage{amsmath}
\usepackage{amsfonts}
\usepackage{amssymb}
\usepackage{color}
\usepackage{hyperref}

\newtheorem{theorem}{Theorem}[section]
\newtheorem{lemma}[theorem]{Lemma}
\newtheorem{prop}[theorem]{Proposition}
\newtheorem{cor}[theorem]{Corollary}

\theoremstyle{definition}
\newtheorem{definition}[theorem]{Definition}
\newtheorem{example}[theorem]{Example}

\theoremstyle{remark}

\newtheorem{remark}[theorem]{Remark}

\numberwithin{equation}{section}
\newcommand{\RR}{{\mathbb R}}

\newcommand{\eps}{\varepsilon}
\newcommand{\out}[1]{\ }
\DeclareMathOperator{\fine}{fine}

\DeclareMathOperator{\mflim}{mf{-}lim}

\DeclareMathOperator{\mfliminf}{mf{-}lim\,inf}

\let\cal=\mathcal
\renewcommand{\phi}{\varphi}

\begin{document}
\title[The Dirichlet problem at the Martin boundary]
{The Dirichlet problem at the Martin boundary of a fine domain}

\author{Mohamed El Kadiri}
\address{Universit\'e Mohammed V
\\D\'epartemnt de Math\'ematiques
\\Facult\'e des Sciences
\\B.P. 1014, Rabat
\\Morocco}
\email{elkadiri@fsr.ac.ma}

\author{Bent Fuglede}
\address{Department of Mathematical Sciences
\\Universitetsparken 5
\\2100 Copenhagen
\\Danmark}
 \email{fuglede@math.ku.dk}

%\subjclass[2000]{
 %\footnote{2000 Mathematics Subject Classification
%32U15, 32U05, 30G12, 31C40}}
\begin{abstract}
We develop the Perron-Wiener-Brelot method of solving the Dirichlet problem
at the Martin boundary of a fine domain in $\RR^n$ ($n\ge2$).
\end{abstract}

\maketitle

\section{Introduction}\label{sec1}

The fine topology on an open set $\Omega\subset\RR^n$ was introduced
by H.\ Cartan in classical potential theory. It is defined as the
smallest topology on $\Omega$ in which every super\-harmonic function on
$\Omega$ is continuous.
Potential theory on a finely open set, for example in $\RR^n$, was introduced
and studied in the 1970's by the second named author
\cite{F1}. The harmonic and super\-harmonic functions and the potentials in this
theory are termed finely [super]harmonic functions and fine
potentials. Generally one distinguishes by the prefix `fine(ly)'
notions in fine potential theory from those in classical potential
theory on a usual (Euclidean) open set. Large parts of classical
potential theory have been extended to fine potential theory.

The integral representation of positive ($=$ nonnegative) finely super\-harmonic
functions by using Choquet's method of extreme points was studied by the first
named author in \cite{El1}, where it was shown that the cone of positive
super\-harmonic functions equipped with the natural topology has a compact base.
This allowed the present authors
in \cite{EF1} to define the Martin compactification and the Martin boundary of
a fine domain $U$ in $\RR^n$. The Martin compactification $\overline U$ of $U$
was defined by injection of $U$ in a compact base of the cone $\cal S(U)$ of
positive finely super\-harmonic functions on $U$. While the Martin boundary of
a usual domain is closed and hence compact, all
we can say in the present setup is that the Martin boundary $\Delta(U)$ of $U$
is a $G_\delta$ subset of the compact Riesz-Martin space
$\overline U=U\cup\Delta(U)$ endowed with the natural topology. Nevertheless
we have defined in \cite{EF1} a suitably measurable Riesz-Martin kernel
$K:U\times\overline U\longrightarrow[0,+\infty]$.
Every function $u\in\cal S(U)$ has an integral representation
$u(x)=\int_{\overline U}K(x,Y)d\mu(Y)$ in terms of a Radon measure $\mu$ on
$\overline U$. This representation is unique if it is required that $\mu$ be
carried by $U\cup\Delta_1(U)$, where $\Delta_1(U)$ denotes the minimal Martin
boundary of $U$, which likewise is a $G_\delta$ in $\overline U$. In this case
of uniqueness we write $\mu=\mu_u$. It was shown that $u$ is a
fine potential, resp.\ an invariant function, if and only if $\mu_u$ is
carried by $U$, resp.\ by $\Delta(U)$. The invariant functions, likewise
studied in \cite{EF1}, generalize the positive harmonic functions in the
classical Riesz decomposition theorem. Finite valued invariant functions are
the same as positive finely harmonic functions.

There is a notion of minimal thinness of a set $E\subset U$ at a point
$Y\in\Delta_1(U)$, and an associated minimal-fine filter $\cal F(Y)$,
which allowed the authors
in \cite{EF1} to obtain a generalization of the classical Fatou-Na{\"\i}m-Doob
theorem.

In a continuation \cite{EF2} of \cite{EF1} we studied sweeping on a subset
of the Riesz-Martin space,
both relative to the natural topology and to the minimal-fine topology on
$\overline U$, and we showed that the two notions of sweeping are identical.
In the present further continuation of \cite{EF1} and \cite{EF2} we investigate
the Dirichlet problem at the Martin boundary of our given fine domain $U$
by adapting the Perron-Wiener-Brelot (PWB) method to the present setup.
It is a complication that there is no Harnack convergence theorem for finely
harmonic functions, and hence the infimum of a sequence of  upper
PWB-functionss on $U$ may equal $-\infty$ precisely
on some nonvoid proper finely closed subset of $U$.
We define resolutivity of a numerical function on $\Delta(U)$ in a standard
way and show that it is equivalent to a weaker, but technically supple
concept called quasi\-resolutivity, which possibly has not been considered
before in the literature (for the classical case where $U$ is Euclidean open).
Our main result implies the corresponding known result for the classical case,
cf.\ \cite[Theorem 1.VIII.8]{Do}.
At the end of Section 3 we obtain analogous results for the case where the
upper and lower PWB-classes are defined in terms of the minimal-fine topology
on $\overline U$ instead of the natural topology. It follows that the two
corresponding concepts of resolutivity are compatible. This result is
possibly new even in the classical case.
A further alternative, but actually equivalent, concept of resolutivity is
discussed in the closing Section 4.

{\bf Notations}: If $U$ is a fine domain in $\Omega$ we
denote by ${\cal S}(U)$ the convex cone of positive finely super\-harmonic
functions on $U$ in the sense of \cite{F1}. The convex cone of fine potentials
on $U$ (that is, the functions in ${\cal S}(U)$ for which every finely
subharmonic minorant is $\le 0$) is denoted by ${\cal P}(U)$. The cone of
invariant functions on $U$ is the
orthogonal band to ${\cal P}(U)$ relative to ${\cal S}(U)$.
By $G_U$ we denote the (fine) Green kernel for $U$, cf.\ \cite{F2}, \cite{F4}.
If $A\subset U$ and $f:A\longrightarrow[0,+\infty]$ one denotes by $R{}_f^A$,
resp.\ ${\widehat R}{}_f^A$, the reduced function, resp.\ the swept function,
of $f$ on $A$ relative to $U$, cf.\ \cite[Section 11]{F1}. If $u\in\cal S(U)$
and $A\subset U$ we may write ${\widehat R}{}_u^A$ for
${\widehat R}{}_f$ with $f:=1_Au$.

\section{The upper and lower PWB$^h$-classes of a function on $\Delta(U)$}
\label{sec2}

We shall study the Dirichlet problem at $\Delta(U)$ relative to a fixed
finely harmonic function $h>0$ on $U$. We denote by $\mu_h$ the
measure on $\Delta(U)$ carried by $\Delta_1(U)$ and representing $h$,
that is $h=\int K(.,Y)d\mu_h(Y)=K\mu_h$. A function $u$ on $U$ (or on
some finely open subset of $U$) is said to be finely
$h$-hyper\-harmonic, finely $h$-super\-harmonic, $h$-invariant, or a fine
$h$-potential, respectively, if it has the form $u=v/h$, where
$v$ is finely hyper\-harmonic, finely super\-harmonic, invariant, or a
fine potential, respectively.

Let $f$ be a function on $\Delta(U)$ with values in ${\overline \RR}$.
A finely $h$-hyper\-harmonic function $u=v/h$ on $U$ is said
to belong to the upper PWB$^h$-class, denoted by
${\overline {\cal U}}{}_f^h$, if $u$ is lower bounded and if
$$\underset{x\to Y,\,x\in U}{\liminf}u(x)\ge f(Y)
\quad\text{ for every }\;Y\in \Delta(U).$$
We define
 $$\dot H{}_f^h=\inf\,{\overline{\cal U}}{}_f^h,
\quad{\overline H}{}_f^h=\widehat{\dot H{}_f^h}
={\widehat\inf}\,{\overline {\cal U}}{}_f^h\;(\le\dot H{}_f^h).$$
Both functions $\overline H{}_f^h$ and $\dot H{}_f^h$ are needed here,
unlike the classical
case where we have the Harnack convergence theorem and hence
$\overline H{}_f^h=\dot H{}_f^h$.
In our setup, $\dot H{}_f^h$
may be neither finely $h$-hyperharmonic nor identically $-\infty$, but only
nearly finely $h$-hyperharmonic on the finely open set
$\{\overline H{}_f^h>-\infty\}$ which can be a nonvoid proper subset of $U$,
see Example \ref{example2.0} below, which also shows that $\Delta(U)$ can be
non-compact.
According to the fundamental convergence theorem \cite[Theorem 11.8]{F1}
${\overline H}{}_f^h$ is finely $h$-hyper\-harmonic on
$\{\overline H{}_f^h>-\infty\}$
and ${\overline H}{}_f^h={\dot H}{}_f^h$ quasieverywhere (q.e.)\! there;
furthermore, since ${\overline{\cal U}}{}_f^h$ is lower directed, there is a
decreasing sequence
$(u_j)\subset\overline{\cal U}{}_f^h$ such that $\inf_ju_j=\dot H{}_f^h$.
Clearly, $\dot H{}_f^h$ is finely u.s.c.\ on all of  $U$, and
${\overline H}{}_f^h$ is finely l.s.c.\ there.

The lower PWB$^h$ class ${\underline{\cal U}}{}_f^h$
is defined by
${\underline {\cal U}}{}_f^h=-{\overline {\cal U}}_{-f}^h$, and we have
$\dot H{}_0^h=0$, hence also
$\overline H{}_0^h={H}{\vrule width 0pt depth 0.2em }_{\kern-0.7em\hbox{.}}\,\,{}_0^h
=\underline H_0^h=0$.
It follows that if $f\ge0$ then
$\dot H_f^h\ge\overline H_f^h\ge0$ and therefore only positive functions
of class $\overline{\cal U}{}_f^h$ need to be considered in the
definition of $\dot H{}_f^h$ and hence of $\overline H_f^h$.
Moreover, $\dot H{}_{\alpha f+\beta}^h=\alpha\dot H{}_f^h+\beta$ and hence
$\overline H_{\alpha f+\beta}^h=\alpha\overline H_f^h+\beta$ for real
constants $\alpha\ge0$ and $\beta$ (when $0$ times $\pm\infty$ is defined
to be $0$).

\begin{example}\label{example2.0} In $\Omega=\RR^n$ with the Green kernel
$G(x,y)=|x-y|^{2-n}$, $n\ge4$, let $\omega\subset\Omega$ be a bounded H\"older
domain such that $\omega$ is irregular with a single
irregular boundary point $z$, cf.\ e.g.\ \cite[Remark 6.6.17]{AG}. Take
$U=\omega\cup\{z\}$. According to \cite[Theorems 1 and 3.1]{Ai} the
Euclidean boundary $\partial\omega$ of $\omega$ is topologically contained
in the Martin boundary $\Delta(\omega)$.
In particular, $z$ is non-isolated as a point of $\Delta(\omega)$.
But $\Delta(U)=\Delta(\omega)\setminus\{z\}$, where $z$ is identified
with $P_z$ (see \cite[Section 3]{EF1}), and since $\Delta(\omega)$ is
compact we infer that $\Delta(U)$ is noncompact.
In $\RR^n$ choose a sequence $(z_j)$ of points of
$\complement{\overline\omega}$ such that $|z_j-z|\le2^{-j}$.
Then $u:=\sum_j2^{-j}G(.,z_j)$ is infinite at $z$, but finite and
harmonic on $\omega$. Furthermore, $u=\sup_ku_k$,
where $u_k:=\sum_{j\le k}2^{-j}G(.,z_j)$ is harmonic and bounded on
$\overline\omega$ ($\subset\RR^n$). It follows that $(u_k)_{|U}$
is of class $\underline{\cal U}{}_f^h$, where $f:=u_{|\Delta(U)}$.
In fact,
$$
\underset{x\to Y,\,x\in U}{\lim}\,u_k(x)=u_k(Y)\le u(Y)=f(Y)
$$
for $Y\in\Delta(U)$ (natural limit on $U\cup\Delta(U)$, or equivalently
Euclidean limit on $\omega\cup((\partial\omega)\setminus\{z\})$. Thus
${H}{\vrule width 0pt depth 0.2em }_{\kern-0.7em\hbox{.}}\,\,{}_f^h
\ge(u_k)_{|U}$, and hence
$$
\underline H{}_f^h(z)\ge
{H}{\vrule width 0pt depth 0.2em }_{\kern-0.7em\hbox{.}}\,\,{}_f^h(z)
\ge\sup_ku_k(z)=u(z)=+\infty.
$$
To show that $\underline H{}_f^h<+\infty$ on $U\setminus\{z\}$ ($\cong\omega$),
let $v\in\underline{\cal U}{}_f^h$. Being upper bounded on the bounded open
set $\omega$, $v$ is subharmonic on $\omega$ by \cite[Theorem 9.8]{F1},
and so is therefore $v-u$.
For any $Y\in\Delta(U)$ ($\cong(\partial\omega)\setminus\{z\}$) we have
$$
\underset{x\to Y,\,x\in\omega}{\limsup}\,v(x)\le f(Y)<+\infty
$$
(also with Euclidean limit), or equivalently
$$
\underset{x\to Y,\,x\in\omega}{\limsup}\,(v(x)-u(x))\le0.
$$
Since $\{z\}$ is polar and $v-u\le v$ is upper bounded, it follows by a
boundary minimum principle that $v-u\le0$, that is, $v\le u$ on $\omega$.
By varying $v\in\underline{\cal U}{}_f^h$ we conclude that
${H}{\vrule width 0pt depth 0.2em }_{\kern-0.7em\hbox{.}}\,\,{}_f^h\le u$
on $\omega\cong U\setminus\{z\}$ and hence by regularization that
$\underline H{}_f^h\le u<+\infty$ on $U\setminus\{z\}$.
Altogether, $\overline H{}_{-f}^h=-\underline H_f^h$ equals
$-\infty$ at $z$, but is finite on $U\setminus\{z\}$.
\end{example}

Henceforth we fix the finely harmonic function $h>0$ on $U$, relative to
which we shall study the Dirichlet problem at $\Delta(U)$. Similarly to
the classical case, cf.\ \cite[p.\ 108]{Do}, we pose the following definition,
denoting by $1_A$ the indicator function of a set $A\subset\Delta(U)$:

\begin{definition}\label{def4.1}
 A subset $A$ of $\Delta(U)$ is said to be $h$-harmonic measure null if
  $\overline H{}_{1_A}^h=0$.
\end{definition}

It will be shown in Corollary \ref{cor6.10c} that $A$ is $h$-harmonic measure
null if and only if $A$ is $\mu_h$-measurable with $\mu_h(A)=0$.

\begin{prop}\label{prop4.2}
{\rm{(a)}} Every countable union of $h$-harmonic measure null sets is $h$-harmonic
measure null.

{\rm{(b)}} A set $A\subset\Delta(U)$ is $h$-harmonic measure null if and only
if there is a finely $h$-super\-harmonic function $u$ (positive if we like)
on $U$ such that $\lim_{x\to Y,\,x\in U}u(x)=+\infty$ for every $Y\in A$.

{\rm{(c)}} If $f:\Delta(U)\to[0,+\infty]$ has $\overline H{}_f^h=0$ then
$\{f>0\}$ is $h$-harmonic measure null.

{\rm{(d)}} If $f:\Delta(U)\to[0,+\infty]$ has $\overline H{}_f^h<+\infty$ then
$\{f=+\infty\}$ is $h$-harmonic measure null.

{\rm{(e)}} If $f,g:\Delta(U)\to[0,+\infty]$ and if $f\le g$ off some
$h$-harmonic measure null set then $\overline H{}_f^h\le\overline H{}_g^h$.
\end{prop}

\begin{proof} We adapt the proof in \cite[p.\ 108, 111]{Do} for the
classical case.

(a) Fix a point $x_0$ of the co-polar subset
$\bigcap_j\{\dot H{}_{1_{A_j}}^h=0\}$ of $U$. For given $\eps>0$ and
integers $j>0$ there are functions $u_j\in\overline{\cal U}{}_{1_{A_j}}^h$ with
$u_j(x_0)<2^{-j}\eps$. It follows that the function
$u:=\sum_ju_j$ is of class $\overline{\cal U}{}_{1_A}^h$ because
$\sum_j1_{A_j}\ge1_A$ on $\Delta(U)$. Consequently,
$\overline H{}_{1_A}^h(x_0)\le\dot H{}_{1_A}^h(x_0)\le u(x_0)<\eps$,
and the positive finely $h$-hyperharmonic function $\overline H{}_{1_A}^h$
therefore equals $0$ at $x_0$ and so indeed everywhere on $U$.

(b) If $\overline H{}_{1_A}=0$ then $\dot H{}_{1_A}^h=0$ q.e., so
we may choose $x_0\in U$ with $\dot H{}_{1_A}^h(x_0)=0$. For integers
$j>0$ there exist positive finely $h$-superharmonic functions
$u_j\in\overline{\cal U}{}_{1_A}^h$ on $U$ such that $u_j(x_0)<2^{-j}\eps$.
The function $u:=\sum_ju_j$ is positive and finely $h$-superharmonic on $U$
because $u(x_0)<+\infty$. Furthermore,
 $\liminf_{x\to Y,\,x\in U}u(x)=+\infty$ for every  $Y\in A$.
Conversely, if there exists a function $u$ as described in (b), we may arrange
that $u\ge0$ after adding a constant. Then
$\eps u\in\overline{\cal U}{}_{1_A}^h$ for every $\eps>0$. It follows that
$\dot H{}_{1_A}^h\le\eps u$ and by varying $\eps$ that
$\dot H{}_{1_A}^h=0$ off the polar set of infinities of $u$, and hence
q.e.\ on $U$. It follows that indeed $\overline H{}_f^h=0$.

(c) For integers $j\ge1$ let $f_j$ denote the indicator function on $U$
for the set $\{f>1/j\}$. Then $0=\overline H{}_f^h\ge\overline H{}_{f_j}^h/j$,
so the sets $\{f>1/j\}$ are $h$-harmonic measure null, and so is by (a) the
union $\{f>0\}$ of these sets.

(d) Choose $x_0\in U$ so that $\dot H{}_f^h(x_0)=\underline H{}_f^h(x_0)$
($<+\infty$) and $u\in\overline{\cal U}{}_f^h$. Then
$\lim_{x\to Y,\,x\in U}u(x)=+\infty$ for every
$Y\in A:=\{f=+\infty\}$. After adding a constant we arrange that the finely
$h$-hyperharmonic function $u$ is positive. According to (b) it follows that
indeed $\overline H{}_{1_A}^h=0$.

(e) Let $v\in\overline{\cal U}{}_g^h$ and let $u$ be a positive
$h$-superharmonic function on $U$ with limit $+\infty$ at every point of
the $h$-harmonic measure null subset $\{f>g\}$ of $\Delta(U)$. Then
$\dot H{}_f^h\le v+\eps u\in\overline{\cal U}{}_f^h$
for every $\eps>0$. Hence $\dot H{}_f^h\le v$ q.e., and so
$\overline H{}_f^h\le v$ everywhere on $U$. By varying $v$ it follows that
$\overline H{}_f^h\le\dot H{}_g^h$ and so indeed by finely l.s.c.\
regularization $\overline H{}_f^h\le\overline H{}_g^h$.
\end{proof}

\begin{prop}\label{prop6.1}
Let $f$ be a function on $\Delta(U)$ with values in ${\overline \RR}$.

{\rm{(a)}} $\dot H{}_f^h
\ge{H}{\vrule width 0pt depth 0.2em }_{\kern-0.7em\hbox{.}}\,\,{}_f^h$
and hence $\overline H{}_f^h
\ge{H}{\vrule width 0pt depth 0.2em }_{\kern-0.7em\hbox{.}}\,\,{}_f^h$ and $\dot H{}_f^h
\ge\underline H{}_f^h$.

{\rm{(b)}} $\overline H{}_f^h
\ge\underline H{}_f^h$ on
$\{{\overline H}{}_f^h>-\infty\}\cup
\{{\underline H}{}_f^h<+\infty\}$.

{\rm{(c)}} If $f$ is lower bounded then
$\overline H{}_f^h(x)=\dot H{}_f^h(x)$ at any point $x\in U$
at which $\dot H{}_f^h(x)<+\infty$. If \,$f\ge0$ on $\Delta(U)$ then
${\overline H}{}_f^h$ $(\ge0)$ is either identically $+\infty$ or
$h$-invariant on $U$.

\end{prop}

\begin{proof} Clearly, ${\overline{\cal U}}{}_f^h$ is lower directed and
$\underline{\cal U}{}_f^h$ is upper directed. The constant function $+\infty$
belongs to ${\overline{\cal U}}{}_f^h$. If \,$+\infty$ is the only function
of class ${\overline{\cal U}}{}_f^h$ then obviously $\dot H{}_f^h=+\infty$
and hence ${\overline H}{}_f^h=+\infty$. In the remaining case it suffices
to consider finely $h$-super\-harmonic functions in the definition of
$\dot H{}_f^h$ and hence of \,${\overline H}{}_f^h$.

(a) Let $u\in\overline{\cal U}{}_f^h$ and $v\in\underline{\cal U}{}_f^h$.
Then $u-v$ is well defined, finely $h$-hyper\-harmonic, and lower bounded
on $U$, and
\begin{eqnarray*}\underset{x\to Y,\,x\in U}{\liminf}(u(x)-v(x))\!\!\!
&\ge&\!\!\!\underset{x\to Y,\,x\in U} {\liminf}u(x)
-\underset{x\to Y,\,x\in U} {\limsup}v(x)\\
&\ge&\!\!\! f(Y)-f(Y)=0
\end{eqnarray*}
if $f(Y)$ is finite; otherwise $\liminf u(x)-\limsup v(x)=+\infty\ge0$,
for if for example $f(Y)=+\infty$ then $\liminf u=+\infty$ whereas
$\limsup v<+\infty$ since $v$ is upper bounded.
By the minimal-fine boundary minimum property given in
\cite[Corollary 3.13]{EF2} together with \cite[Proposition 3.5]{EF2}
applied to the finely super\-harmonic function $hu-hv$
(if $\ne+\infty$) it then follows that $u-v\ge0$, and hence
$u\ge v$. By varying $u$ and $v$ in either order we obtain
$\dot H{}_f^h
\ge{H}{\vrule width 0pt depth 0.2em }_{\kern-0.7em\hbox{.}}\,\,{}_f^h$.
Since
${H}{\vrule width 0pt depth 0.2em }_{\kern-0.7em\hbox{.}}\,\,{}_f^h
=\sup\underline{\cal U}{}_f^h$
is finely l.s.c.\ it follows that
${\overline H}{}_f^h
\ge{H}{\vrule width 0pt depth 0.2em }_{\kern-0.7em\hbox{.}}\,\,{}_f^h$,
and similarly $\dot H{}_f^h
\ge{\underline H}{}_f^h$.

(b) Consider any point $x_0$ of the finely open set
$V=\{\overline H{}_f^h>-\infty\}$.
Since ${\overline H}{}_f^h$ is finely $h$-hyper\-harmonic
and hence finely continuous on $V$
we obtain by (a)
$$
{\overline H}{}_f^h(x_0)
=\underset{x\to x_0,\,x\in V}{\fine\lim}\,{\overline H}{}_f^h(x)
\ge\underset{x\to x_0,\,x\in V\setminus E}{\fine\lim\sup}\,
{H}{\vrule width 0pt depth 0.2em }_{\kern-0.7em\hbox{.}}\,\,{}_f^h
={\underline H}{}_f^h(x_0).
$$
The case $x_0\in\{\underline H{}_f^h<+\infty\}$ is treated similarly,
or by replacing $f$ with $-f$.

(c) The former assertion reduces easily to the case $f\ge0$, whereby
$\overline{\cal U}{}_f^h$ consists of positive functions. We may assume that
$\overline H{}_f^h\not\equiv+\infty$, and hence
$h\overline{\cal U}{}_f^h\subset\cal S(U)$.
Consider the cover of $U$ by the finely open sets $V_k$ from
\cite[Lemma 2.1 (c)]{EF2}. Then $h\overline{\cal U}{}_f^h$ is a Perron family
in the sense of \cite[Definition 2.2]{EF2}. It therefore follows by
\cite[Theorem 2.3]{EF2} that indeed
${\overline H}{}_f^h=\widehat{\inf}\,\overline{\cal U}{}_f^h$ is
$h$-invariant, and that $\overline H{}_f^h(x)=\dot H{}_f^h(x)$ at any point
$x\in U$ at which $\dot H{}_f^h(x)<+\infty$.
\end{proof}

\begin{prop}\label{prop6.2} Let $f,g$ be two functions on $\Delta(U)$ with
values in $\overline\RR$.

{\rm{1.}} If $f+g$ is well defined everywhere on $\Delta(U)$ then
the inequality $\dot H{}_{f+g}^h\le\dot H{}_f^h+\dot H{}_g^h$
holds at each point of $U$ where $\dot H{}_f^h+\dot H{}_g^h$ is
well defined.

{\rm{2.}} If $(f+g)(Y)$ is defined arbitrarily at points $Y$ of
undetermination then the inequality
$\overline H{}_{f+g}^h\le\overline H{}_f^h+\overline H{}_g^h$
holds everywhere on $\{\overline H{}_f^h,\overline H{}_g^h>-\infty\}$.

{\rm{3.}} For any point $x\in U$ we have $\dot H{}_f^h(x)<+\infty$
if and only if $\dot H{}_{f\vee0}^h(x)<+\infty$.

{\rm{4.}} Let $(f_j)$ be an increasing sequence of lower bounded functions
$\Delta(U)\longrightarrow\,]-\infty,+\infty]$. Writing $f=\sup_jf_j$ we have
$\overline H{}_f^h=\sup_j\overline H{}_{f_j}^h$ and
$\dot H{}_f^h=\sup_j\dot H{}_{f_j}^h$.
\end{prop}

\begin{proof} For 1., 2., and 4.\ we proceed much as in
\cite[1.VIII.7, Proof of (c), (b), and (e)]{Do}.
For Assertion 1., consider any two functions $u\in\overline{\cal U}{}_f^h$
and $v\in\overline{\cal U}{}_g^h$.
Then $u+v\in\overline{\cal U}{}_{f+g}^h$ and hence $\dot H{}_{f+g}^h\le u+v$.
By varying $v$ it follows that $\dot H{}_{f+g}^h\le u+\dot H{}_g^h$ on
$\{\dot H{}_g^h>-\infty\}$.
By varying $u$ this leads to
$\dot H{}_{f+g}^h\le\dot H{}_f^h+\dot H{}_g^h$ whereever the sum is well defined
on $\{\dot H{}_g^h>-\infty\}$.
By interchanging $u$ and $v$ we infer that
$\dot H{}_{f+g}^h\le\dot H{}_f^h+\dot H{}_g^h$ altogether holds whereever the
sum is well defined on $\{\dot H{}_f^h,\dot H{}_g^h>-\infty\}$.
On the residual set $\{\dot H{}_f^h=-\infty\}\cup\{\dot H{}_g^h=-\infty\}$
it is easily seen that
$\dot H{}_{f+g}^h=-\infty=\dot H{}_f^h+\dot H{}_g^h$ whereever the sum is well
defined.

For Assertion 2., suppose first that $f,g<+\infty$ (and so $f+g$ is well
defined). In the proof of Assertion 1.\ we had
$\dot H{}_{f+g}^h\le u+\dot H{}_g^h$ for $\dot H{}_g^h>-\infty$, which is
satisfied on $\{\overline H{}_g^h>-\infty\}$. It follows that
$\overline H{}_{f+g}^h\le u+\overline H{}_g^h$ there, and hence that
$\overline H{}_{f+g}^h\le\overline H{}_f^h+\overline H{}_g^h$ there (whereever
well defined). In the general case
define functions $f_0<+\infty$, resp.\ $g_0<+\infty$, which equal $f$,
resp.\ $g$, except on the set $\{f=+\infty\}$, resp.\ $\{g=+\infty\}$.
We may assume that these exceptional sets are $h$-harmonic measure null,
for if e.g.\ $\{f=+\infty\}$ is not $h$-harmonic measure null then
$\overline H{}_f^h\equiv=+\infty$ by Proposition \ref{prop4.2} (d),
in which case 1.\ becomes trivial.
It therefore follows in view of Proposition
\ref{prop4.2} (a), (e) that $f+g=f_0+g_0$ off the $h$-harmonic measure
null set $\{f=+\infty\}\cup\{g=+\infty\}$ and hence by Proposition
\ref{prop4.2} (e) that
$$
\overline H{}_{f+g}^h=\overline H{}_{f_0+g_0}^h
\le\overline H{}_{f_0}^h+\overline H{}_{g_0}^h
\le\overline H{}_f^h+\overline H{}_g^h$$
on the finely open set
$\{\overline H{}_f^h>-\infty\}\cap\{\overline H{}_g^h>-\infty\}$,
the second inequality because $f_0\le f$ and $g_0\le g$.

For Assertion 3., let $x\in U$ be given, and suppose that
$\dot H{}_f^h(x)<+\infty$.
There is then $u\in\overline{\cal U}{}_f^h$ with $u(x)<+\infty$, $u$ being
finely $h$-superharmonic $\ge-c$ for some constant $c\ge0$. It follows that
$u+c\in\overline{\cal U}{}_{f\vee0}^h$ and hence
$\dot H{}_{f\vee0}^h(x)\le u(x)+c<+\infty$. The converse follows from
$\dot H{}_f^h\le\dot H{}_{f\vee0}^h$.

Assertion 4.\ reduces easily to the case of positive functions $f_j$.
Consider first the case of $\overline H$. Then
${\overline H}{}_f^h$ and each ${\overline H}{}_{f_j}^h$ are
positive and hence finely $h$-hyper\-harmonic by Proposition \ref{prop6.1} (c).
The inequality ${\overline H}{}_f^h\ge\sup_j{\overline H}{}_{f_j}^h$ is obvious,
and we may therefore assume that the positive finely $h$-hyper\-harmonic
function $\sup_j{\overline H}{}_{f_j}^h$ is not identically $+\infty$, and
therefore is $h$-invariant, again according to
Proposition \ref{prop6.1} (c). Denote $E_j$ the polar subset
$\{\overline H{}_{f_j}^h<\dot H{}_{f_j}^h\}$
%\cup\{\overline H{}_{f_j}^h=-\infty\}$
of $U$ and write $E:=\bigcup_jE_j$ (polar). For a fixed $x\in U\setminus E$
and for given $\eps>0$ choose functions $u_j\in{\overline{\cal U}}{}_{f_j}^h$
so that
\begin{eqnarray}u_j(x)\!\!\!
&<&\!\!\!\dot H{}_{f_j}^h(x)+2^{-j}\eps=\overline H{}_{f_j}^h(x)+2^{-j}\eps.
\end{eqnarray}
In particular, $u_j$ is finely $h$-super\-harmonic.
 Define a finely $h$-hyper\-harmonic function $u$ by
\begin{eqnarray}u\!\!\!
&=&\!\!\!\sup_j{\overline H}{}_{f_j}^h
+\sum_j(u_j-{\overline H}{}_{f_j}^h)
\ge{\overline H}{}_{f_k}^h+(u_k-{\overline H}{}_{f_k}^h)=u_k
\end{eqnarray}
for any index $k$. Then
$$\liminf_{x\to Y,\,x\in U}u(x)
\ge\liminf_{x\to Y,\,x\in U}u_k(x)
\ge f_k(Y)$$
for every $Y\in\Delta(U)$ and every index $k$.
Thus $u\in{\overline{\cal U}}{}_f^h$ and
${\overline H}{}_f^h\le{\dot H}{}_f^h\le u$.
In particular, by the former equality (2.2) and by (2.1),
\begin{eqnarray}{\overline H}{}_f^h(x)\!\!\!
&\le&\!\!\!u(x)
\le\sup_j{\overline H}{}_{f_j}^h(x)+\eps,
\end{eqnarray}
and hence the finely $h$-hyper\-harmonic function ${\overline H}{}_f^h$ is
finely $h$-super\-harmonic.
Because $\sup_j{\overline H}{}_f^h$ is $h$-invariant and majorized by
${\overline H}{}_f^h$, the function
${\overline H}{}_f^h-\sup_j {\overline H}{}_{f_j}^h$ is finely
$h$-super\-harmonic $\ge0$ on $U$ by \cite [Lemma 2.2]{EF1},
and $\le\eps$ at $x$. For $\eps\to0$ we obtain the remaining inequality
${\overline H}{}_f^h\le\sup_j {\overline H}{}_{f_j}^h$.

In the remaining case of $\dot H$ we have
$\sup_j\dot H{}_{f_j}^h(x)\le\dot H{}_f^h$. For any point $x\in U$ at which
$\sup_j\dot H{}_{f_j}^h(x)<\dot H{}_f^h$ and for any $\eps>0$ choose functions
$u_j\in{\overline{\cal U}}{}_{f_j}^h$ so that
\begin{eqnarray}u_j(x)\!\!\!
&<&\!\!\!\dot H{}_{f_j}^h(x)+2^{-j}\eps.
\end{eqnarray}
Proceed as in the above case of $\overline H$ by defining the finely
$h$-hyperharmonic function by (2.2), replacing throughout $\bar H$
by $\dot H$. Corresponding to (2.3) we now end by
$$
\dot H{}_f^h(x)\le u(x)\le\sup_j{\dot H}{}_{f_j}^h(x)+\eps,
$$
from which the remaining inequality
$\dot H{}_f^h(x)\le\sup_j{\dot H}{}_{f_j}^h(x)$ follows for $\eps\to0$.
\end{proof}

\section{$h$-resolutive and $h$-quasiresolutive functions}\label{sec3}

\begin{definition}\label{def6.9}
A function $f$ on $\Delta(U)$ with values in ${\overline \RR}$ is said
to be $h$-resolutive if $\overline H{}_f^h=\underline H{}_f^h$ on $U$
and if this function, also denoted by $H_f^h$,
is neither identically $+\infty$ nor identically $-\infty$.
\end{definition}

It follows that $H_f^h$ is finely $h$-harmonic on the finely open set
$\{\overline H{}_f^h>-\infty\}\cap\{\underline H{}_f^h>+\infty\}
=\{-\infty<H_f^h<+\infty\}$.

For any function $f:\Delta(U)\longrightarrow\RR$ we consider the following
two subsets of $U$:
$$
E_f^h=\{\overline H{}_f^h=-\infty\}\cup
\{\underline H{}_f^h=+\infty\}\cup
\{\overline H{}_f^h\ne\underline H{}_f^h\},
$$
$$
P_f^h=\{\dot H{}_f^h>\overline H{}_f^h>-\infty\}
\cup\{{H}{\vrule width 0pt depth 0.2em }_{\kern-0.7em\hbox{.}}\,\,{}_f^h
<\underline H{}_f^h<+\infty\},
$$
of which $P_f^h$ is polar.

\begin{definition}\label{def2.3} A function $f$ on $\Delta(U)$
with values in ${\overline \RR}$ is said to be
$h$-quasi\-resolutive if $E_f^h$ is polar, or equivalently if the relations
$\overline H{}_f^h>-\infty$,
$\underline H{}_f^h<+\infty$, and
$\overline H{}_f^h=\underline H{}_f^h$ hold quasieverywhere on $U$.
\end{definition}

When $f$ is $h$-quasi\-resolutive on $U$ the functions $\overline H{}_f^h$
and $\underline H{}_f^h$ are finely $h$-hyperharmonic and finely
$h$-hypoharmonic, respectively, off the polar set
$\{\overline H{}_f^h=-\infty\}\cup\{\underline H{}_f^h=+\infty\}$, and they
are actually equal and hence finely $h$-harmonic off the smaller
polar set $E_f^h$. We then denote
by $H_f^h$ the common restriction of $\overline H{}_f^h$ and
$\underline H{}_f^h$ to $U\setminus E_f^h$.
Since $\overline H{}_f^h$ is finely $h$-hyperharmonic on $U\setminus E_f$, and
$\underline H{}_f^h$ is finely $h$-hypoharmonic there, it
follows by Proposition \ref{prop6.1} (a) that the equalities
$$
\overline H{}_f^h=\dot H{}_f^h
={H}{\vrule width 0pt depth 0.2em }_{\kern-0.7em\hbox{.}}\,\,{}_f^h=
\underline H{}_f^h
$$
hold q.e.\ on $U\setminus E_f^h$ and hence q.e.\ on $U$.

\begin{lemma}\label{lemma3.3} Every $h$-resolutive function $f$ is
$h$-quasi\-resolutive.
\end{lemma}

\begin{proof} The sets $E_+:=\{H_f^h=+\infty\}$ and $E_-:=\{H_f^h=-\infty\}$
are finely closed and disjoint. For any fine component $V$ of
$U\setminus E_-$ such that $V\cap E_+$ is nonpolar we have $V\cap E_+=V$,
that is, $V\subset E_+$, because $\overline H{}_f^h$ is finely $h$-hyperharmonic
on $V$. Denote by $W$ the union of these fine components $V$, and by $W'$ the
(countable) union of the remaining fine components $V'$ of $U\setminus E_-$.
Then $W\subset E_+$ whereas the set $P:=W'\cap E_+$ is polar along with each
$V'\cap E_+$. Since $E_+\cap E_-=\varnothing$ we obtain
$$
E_+=(U\setminus E_-)\cap E_+=(W\Cup W')\cap E_+
=(W\cap E_+)\Cup(W'\cap E_+)=W\Cup P,
$$
$\Cup$ denoting disjoint union.
Now, $(U\setminus P)\cap E_+=E_+\setminus P$ is finely closed relatively to
the nonvoid fine domain $U\setminus P$ (cf.\ \cite[Theorem 12.2]{F1}), but also
finely open, being equal to $W$ as seen from the above display. Thus either
$W=U\setminus P$ or $W=\varnothing$. But $W=U\setminus P$ would imply $E_+=U$,
contradicting $H_f^h\not\equiv+\infty$, and so actually $E_+=P$ (polar).
Similarly (or by replacing $f$ with $-f$) it is shown that $E_-$ is polar,
and so $f$ is $h$-quasiresolutive because
$\overline H{}_f^h=\underline H{}_f^h$ even holds everywhere on $U$.
\end{proof}

In view of Lemma \ref{lemma3.3} an $h$-quasi\-resolutive function
$f$ is $h$-resolutive if and only if $E_f^h=\varnothing$ (any polar subset
of $\Delta(U)$ being a proper subset). This implies that 1.\ and 2.\ in
the following proposition remain valid with `$h$-quasi\-resolutive' replaced
throughout by `$h$-resolutive'.
It will be shown in Corollary \ref{cor6.3b} that $h$-resolutivity and
$h$-quasi\-resolutivity are actually identical concepts.

\begin{prop}\label{prop6.3} Let $f,g:\Delta(U)\longrightarrow\overline{\RR}$
be $h$-quasi\-resolutive. Then

{\rm{1.}} For $\alpha\in\RR$ we have $E_{\alpha f}^h\subset E_f^h$ and hence
$\alpha f$ is $h$-quasi\-resolutive.
Furthermore, $H{}_{\alpha f}^h=\alpha H{}_f^h$ on $U\setminus E_f^h$.

{\rm{2.}} If $f+g$ is defined arbitrarily at points of undetermination then
$E_{f+g}^h\subset E_f^h\cup E_g^h$ and hence $f+g$ is $h$-quasi\-resolutive.
Furthermore, $H{}_{f+g}^h=H{}_f^h+H{}_g^h$ on $U\setminus(E_f^h\cup E_g^h)$.

{\rm{3.}} $E_{f\vee g}^h,E_{f\wedge g}^h$ $\subset E_f^h\cup E_g^h\cup P_f^h$ and
hence $f\vee g$ and $f\wedge g$ are $h$-quasi\-resolutive.
If for example $H_f^h\vee H_g^h\ge0$
then $H{}_{f\vee g}^h=(1/h){\widehat R}{}_{(H_f^h\vee H{}_g^h)h}$ on
$U\setminus(E_f^h\cup E_g^h\cup P_f^h)$.
\end{prop}

\begin{proof} For Assertion 1., consider separately the cases
$\alpha>0$, $\alpha<0$, and $\alpha=0$.
For 2.\ and 3. we proceed as in \cite[1.VIII.7 (d)]{Do}.
For 2.\ we have
\begin{eqnarray}
H_f^h+H_g^h\ge\overline H{}_{f+g}^h\ge\underline H{}_{f+g}^h\ge H_f^h+H_g^h
\end{eqnarray}
on $U\setminus(E_f^h\cup E_g^h)$, the first inequality by 2.\ in
Proposition \ref{prop6.2},
the third inequality by replacing $f$ with $-f$ in the
first inequality, and the second inequality holds by Proposition
\ref{prop6.1} (b) on
$$
\{\overline H_{f+g}^h>-\infty\}
\supset\{\overline H_f^h>-\infty\}\cap\{\overline H_g^h>-\infty\}
\supset U\setminus(E_f\cup E_g).
$$
Thus equality prevails on $U\setminus(E_f\cup E_g)$ (and hence q.e.\ on $U$)
in both of these inclusion relations. It follows that
$$
\{\overline H{}_{f+g}^h=-\infty\}
\subset\{\overline H{}_f^h=-\infty\}\cup\{\overline H{}_g^h=-\infty\}
\subset E_f^h\cup E_g^h.
$$
and similarly
$\{\underline H_{f+g}^h=-\infty\}\subset E_f^h\cup E_g^h$. Finally, by (3.1)
with equality throughout,
$$\{\overline H{}_{f+g}^h\ne\underline H{}_{f+g}^h\}
\subset\{\overline H{}_f^h\ne\underline H{}_f^h\}
\cup\{\overline H{}_g^h\ne\underline H{}_g^h\}
\subset E_f^h\cup E_g^h.
$$
Altogether, $E_{f+g}^h\subset E_f^h\cup E_g^h$,
and so $f+g$ is indeed $h$-quasi\-resolutive along with $f$ and $g$.

For the notation in the stated equation in 3., see \cite[Definition 11.4]{F1}.
Since $f\wedge g=-[(-f)\vee(-g)]$ and $f\vee g=[(f-g)\vee0]+g$  it follows by
1.\ and 2.\ that 3.\ reduces to $E_{f\vee0}^h\subset E_f\cup P_f^h$,
which implies the $h$-quasi\-resolutivity of $f^+=f\vee0$
and the stated expression for $H_{f\vee g}^h$ with $g=0$.
For given $x\in U\setminus(E_f^h\cup P_f^h)$ and integers $j>0$
choose $u_j\in\overline{\cal U}{}_f^h$ with
$u_j(x)\le\dot H{}_f^h(x)+2^{-j}=\overline H{}_f^h(x)+2^{-j}$.
The series $\sum_{j=k}^\infty(u_j-H_f^h)$ of
positive finely $h$-super\-harmonic functions on $U\setminus(E_f^h\cup P_f^h)$
($u_j$ being likewise restricted to $U\setminus(E_f^h\cup P_f^h)$) has a positive
finely $h$-super\-harmonic sum, finite at $x$.
Recall that $H_f^h$ is  defined and finely $h$-harmonic on $U\setminus E_f^h$
and in particular on $U\setminus(E_f^h\cup P_f^h)$. Consequently,
$H_f^h\vee0$ is finely $h$-subharmonic (and positive)
on $U\setminus(E_f^h\cup P_f^h)$ and majorized there by $\overline H{}_{f\vee0}^h$,
which is finite valued on $U\setminus(E_f^h\cup P_f^h)$ by 3.\ in Proposition
\ref{prop6.2} because $\overline H{}_f^h<+\infty$ on
$U\setminus(E_f^h\cup P_f^h)$ and because $\overline H{}_f^h=\dot H{}_f^h$
there. It follows by \cite[Theorem 11.13]{F1}, applied with $f$ replaced by
$h\overline H{}_f^h\vee0$ on $U$,
which is finely subharmonic on $U\setminus(E_f^h\cup P_f^h)$,
that $\frac1h{\widehat R}{}_{h\overline H{}_f\vee0}^h$ (sweeping relative to $U$)
is finely $h$-harmonic on $U\setminus(E_f^h\cup P_f^h)$,
being majorized there by
$\overline H{}_f^h\vee0\le\overline H{}_{f\vee0}^h<+\infty$.
The positive function
\begin{eqnarray}
\frac1h{\widehat R}{}_{h\overline H{}_f\vee0}^h
+\sum_{j=k}^\infty(u_j-H_f^h)
\end{eqnarray}
restricted to $U\setminus(E_f^h\cup P_f^h)$ is therefore finely
$h$-super\-harmonic.
Moreover, this positive finely  $h$-super\-harmonic function on
$U\setminus(E_f^h\cup P_f^h)$ majorizes $u_k\in{\overline{\cal U}}{}_f^h$ there
(being $\ge\frac1h{\widehat R}{}_{h\overline H{}_f\vee0}^h+(u_k-H_f^h)\ge u_k$ there),
and this majorization remains in force after extension by fine continuity
to $U$, cf.\ \cite[Theorem 9.14]{F1}. Thus the extended positive function (3.2)
belongs to $\overline{\cal U}{}_{f\vee 0}^h$. For $k\to\infty$ it follows that
${\overline H}{}_{f\vee 0}^h
\le\frac1h{\widehat R}{}_{h\overline H{}_{f\vee 0}^h}$
on $U$. On the other hand, $\underline H{}_{f\vee0}^h$ majorizes both
 $\underline H{}_f^h$ and $0$,
so $\underline H{}_{f\vee0}^h\ge\frac1h{\widehat R}{}_{h\underline H{}_f^h\vee0}
=\frac1h{\widehat R}{}_{h\overline H{}_f^h\vee0}$ on $U$, the equality because
$\underline H{}_f^h=\overline H{}_f^h$ on $U\setminus(E_f^h\cup P_f^h)$ and hence
q.e.\ on $U$.
It follows that
$${\overline H}{}_{f\vee0}^h\le\frac1h{\widehat R}{}_{h\overline H{}_f^h\vee0}
=\frac1h{\widehat R}{}_{h\underline H{}_f^h\vee0}
\le\underline H{}_{f\vee0}^h\le{\overline H}{}_{f\vee0}^h<+\infty
$$ because
$h\overline H{}_f^h\vee0=h\underline H{}_f^h\vee0$
on $U\setminus(E_f^h\cup P_f^h)$ and hence q.e.\ on $U$. (The last inequality
in the above display follows by Proposition \ref{prop6.1} (b) because
$f\vee0>-\infty$.)
Since $\overline H{}_{f\vee0}^h\ge0\ge-\infty$ we conclude that $f\vee0$
indeed is $h$-resolutive, resp.\ $h$-quasi\-resolutive, and that
$E_{f\vee0}^h\subset E_f^h\cup P_f^h$ and
$H_{f\vee0}^h=\frac1h{\widehat R}{}_{h\overline H{}_f^h\vee0}$ on $E_f^h\cup P_f^h$.
\end{proof}

A version of Proposition \ref{prop6.3} for $h$-resolutive
functions instead of $h$-quasi\-resolutive functions will
of course follow when the identity of $h$-resolutivity and
$h$-quasi\-resolutivity has been established in Corollary \ref{cor6.3b}.
Before that, we do however need the following step in that direction,
based on Proposition \ref{prop6.2}.

\begin{lemma}\label{lemma6.3c} Let $f$ be an $h$-quasi\-resolutive function
on $\Delta(U)$. If $f^+$ and $f^-$ are $h$-resolutive then so is $f$,
and the function
$H_f^h=\overline H{}_f^h=\underline H{}_f^h$ on $U$ is finite valued.
\end{lemma}

\begin{proof} According to 3.\ in Proposition \ref{prop6.3}, $f^+$ and $f^-$
are $h$-quasi\-resolutive (besides being $h$-resolutive), and the functions
$H_{f+}^h:=\overline H{}_{f^+}^h=\underline H{}_{f^+}^h$
(defined on $\{\overline H{}_{f_+}^h>-\infty\}=U$ since $f^+\ge0$)
and similarly $H_{f-}^h:=\overline H{}_{f^-}^h=\underline H{}_{f^-}^h$
are therefore finite valued. Since $-f^-\le f\le f^+$ it follows that
$-\infty<-H_{f^-}^h\le\underline H{}_f^h,\overline H{}_f^h\le H_{f^+}^h<+\infty$.
Applying 2.\ in Proposition \ref{prop6.2} to the sums
$f=f^++(-f)^+=f^+-f^-$ and $-f=f^--f_+$, which are well defined on $\Delta(U)$,
we obtain
$$
\overline H{}_f^h\le H_{f^+}^h-H_{f^-}^h\le\underline H{}_f^h
$$
on all of $U$, and hence $\overline H{}_f^h=\underline H{}_f^h$ there because
$\overline H{}_f^h\ge\underline H{}_f^h$ on all of $U$, again by
Proposition \ref{prop6.1} (b) since we have seen that for example
$\overline H{}_f^h>-\infty$.
\end{proof}

\begin{cor}\label{cor4.8} Let $(f_j)$ be an increasing sequence of lower
bounded $h$-resolutive, resp.\  $h$-quasi\-resolutive functions
$\Delta(U)\longrightarrow\,]-\infty,+\infty]$, and let $f=\sup_jf_j$.
If $\overline H{}_f^h\not\equiv+\infty$ then $f$ is
$h$-resolutive, resp.\ $h$-quasi\-resolutive.
 \end{cor}

\begin{proof} By adding a constant to $f$ we reduce the claim to the case
$f_j\ge0$. For every $j$ we have
$\underline H{}_f^h\ge\underline H{}_{f_j}^h=\overline H{}_{f_j}^h$,
by Proposition \ref{prop6.1} (b) because $\overline H{}_{f_j}^h>-\infty$. Hence
$\underline H{}_f^h\ge\sup_j\overline H{}_{f_j}^h=\overline H{}_f^h$
according to 4.\ in Proposition \ref{prop6.2}.
By definition of $\underline H{}_f^h$ we have at any point $x_0\in U$
$$
\underline H{}_f^h(x_0)=\underset{x\to x_0,\,x\in U}{\fine\lim\sup}\,
{H}{\vrule width 0pt depth 0.2em }_{\kern-0.7em\hbox{.}}\,\,{}_f^h(x)
\le\underset{x\to x_0,\,x\in U}{\fine\lim\sup}\,
\overline H{}_f^h(x)
=\overline H{}_f^h(x_0)
$$
according to Proposition \ref{prop6.1} (a), $\overline H{}_f^h$ being finely
$h$-hyper\-harmonic by Proposition \ref{prop6.1} (c). By Proposition
\ref{prop6.1} (b) we have
$\underline H{}_f^h\le\overline H{}_f^h$ on $\{\overline H{}_f^h>-\infty\}=U$,
and we conclude that $\underline H{}_f^h=\overline H{}_f^h$. By hypothesis
this finely hyperharmonic function on $U$ is finite q.e.\ on $U$, and in
particular this positive function is not identically $+\infty$.
Consequently, $f$ is indeed $h$-resolutive, resp.\ $h$-quasi\-resolutive.
\end{proof}

Recall that $\mu_h$ denotes the unique measure on $\overline U$
carried by $\Delta_1(U)$ and representing $h$,
that is, $h=K\mu_h=\int K(.,Y)d\mu_h(Y)$.

\begin{prop}\label{prop6.4} For any $\mu_h$-measurable subset $A$
of $\Delta(U)$ the indicator function $1_A$ is $h$-resolutive, and
\begin{eqnarray}H{}_{1_A}^h\!\!\!
&=&\!\!\!\frac{1}{h}\int_AK(.,Y)d\mu_h(Y)=\frac1h\widehat R{}_h^A
\end{eqnarray}
on $U$.
In particular, the constant function $1$ on $\Delta(U)$ is
$h$-resolutive and $H_1^h=1$.
\end{prop}

\begin{proof} Because $h=K\mu_h$ and because $\mu_h$ is carried by
$\Delta_1(U)$ we have by \cite[Theorem 3.10 and Proposition 3.9]{EF2}
$$
\widehat R{}_h^A=\widehat R{}_{K\mu_h}^A
=\int_{\Delta_1(U)}\widehat R{}_{K(.,Y)}^Ad\mu_h(Y)
=\int K(.,Y)1_A(Y)d\mu_h(Y).
$$
Consider any finely $h$-hyper\-harmonic function $u=v/h\ge0$ on $U$ such that
$u\ge1$ on some open set $W\subset\overline U$ with $W\supset A$. Then
$u\in\overline{\cal U}{}_{1_A}^h$ and hence
$u\ge\dot H{}_{1_A}^h\ge\overline H{}_{1_A}^h$.
By varying $W$ it follows by \cite[Definition 2.4]{EF2} that
$\frac1h\widehat R{}_h^A\ge\overline H{}_{1_A}^h$.
We have
\begin{eqnarray}\frac1h\int K(.,Y)1_A(Y)d\mu_h(Y)
=\frac1h \widehat R{}_h^A\ge{\overline H}{}_{1_A}^h.
\end{eqnarray}
Applying this inequality to the $\mu_h$-measurable set $\Delta(U)\setminus A$
in place of $A$ we obtain
\begin{eqnarray}\frac{1}{h}\int K(.,Y)1_{\Delta(U)\setminus A}(Y)d\mu_h(Y)
\ge {\overline H}{}_{1_{\Delta(U)\setminus A}}^h.
\end{eqnarray}
By adding the left hand, resp.\ right hand, members of (3.4) and (3.5)
this leads by 2.\ in Proposition \ref{prop6.2} to
\begin{eqnarray}1=\frac1h\int K(.,Y)d\mu_h(Y)
\ge{\overline H}{}_{1_A}^h+{\overline H}{}_{1_{\Delta(U)\setminus A}}^h
\ge{\overline H}{}_1^h=1 .
\end{eqnarray}
Thus equalities prevail throughout in (3.4), (3.5), and (3.6). It follows
altogether that
\begin{eqnarray*}\underline H{}_{1_A}^h\!\!\!
&=&-\overline H{}_{-1_A}^h=1-\overline H{}_{1-1_A}^h
=1-\overline H{}_{1_{\Delta(U)\setminus A}}^h\\
&=&\!\!\!\overline H{}_{1_A}^h
=\frac1h\widehat R{}_h^A=\frac{1}{h}\int_AK(.,Y)d\mu_h(Y),
\end{eqnarray*}
so that indeed $1_A$ is $h$-resolutive and (3.3) holds.
\end{proof}

For any function $f:\Delta(U)\longrightarrow\overline\RR$ we define
$f(Y)K(x,Y)=0$ at points $(x,Y)$ where $f(Y)=0$ and $K(x,Y)=+\infty$.
If $f$ is $\mu_h$-measurable then so is $Y\longmapsto f(Y)K(x,Y)$ for each
$x\in U$ because $K(x,Y)>0$ is $\mu_h$-measurable (even l.s.c.) as a function
of $Y\in\Delta(U)$ according to \cite[Proposition 2.2 (i)]{EF1}.

\begin{prop}\label{prop6.5} Let $f$ be a $\mu_h$-measurable
lower bounded function on $\Delta(U)$. Then
$$
{\overline H}{}_f^h=\frac{1}{h}\int f(Y)K(.,Y) d\mu_h(Y)>-\infty,
$$
and ${\overline H}{}_f^h$ is either identically $+\infty$ or the sum of
an $h$-invariant function and a constant $\le0$.
Furthermore $f$ is $h$-quasi\-resolutive if and only if $f$ is $h$-resolutive,
and that holds if and only if
$\frac{1}{h}\int f(Y)K(.,Y)d\mu_h(Y)<+\infty$
q.e.\ on $U$, or equivalently: everywhere on $U$.
In particular, every bounded $\mu_h$-measurable function
$f:\Delta(U)\longrightarrow\RR$ is $h$-resolutive.
\end{prop}

\begin{proof} Consider first the case of a positive $\mu_h$-measurable
function $f$. Then $f$ is the pointwise supremum of an
increasing sequence of positive $\mu_h$-measurable step
functions $f_j$ (that is, finite valued functions $f_j$ taking only finitely
many values, each finite and each on some $\mu_h$-measurable set; in other
words: affine combinations of indicator functions of $\mu_h$-measurable sets).
For any index $j$ it follows by Proposition \ref{prop6.4} and by 1.\ and 2.\
in Proposition \ref{prop6.3} (the latter extended to finite sums
and with `$h$-resolutive' throughout in place of
`$h$-quasi\-resolutive', cf.\ the paragraph preceding Proposition
\ref{prop6.3}) that each $f_j$ is $h$-resolutive and that
$$
H_{f_j}^h=\frac{1}{h}\int f_j(Y)K(.,Y)d\mu_h(Y)
$$
on $U$, whence
\begin{eqnarray*}0\!\!\!
&\le&\!\!\!\frac{1}{h}\int f(Y)K(.,Y)d\mu_h(Y)\\
&=&\!\!\!\frac1{h}\sup_j\int f_j(Y)K(.,Y)d\mu_h(Y)
=\sup_j H_{f_j}^h={\overline H}{}_f^h
\end{eqnarray*}
by 4.\ in Proposition \ref{prop6.2}.
For a general lower bounded $\mu_h$-measurable function $f$ on $\Delta(U)$
there is a constant $c\ge0$ such that
$g:=f+c\ge0$ and hence $\overline {}H_g^h=\overline {}H_f^h+c\ge0$.
It follows that
$$
{\overline H}{}_f^h=\frac1h\int g(Y)K(.,Y)d\mu_h(Y)-c
=\frac1h\int f(Y)K(.,Y)d\mu_h(Y)>-\infty
$$
and hence by Proposition \ref{prop6.1} (c) applied to $g$ that
${\overline H}{}_f^h=\overline H{}_g^h-c$ has the asserted form.

Next, consider a bounded $\mu_h$-measurable function $f$ on $\Delta(U)$.
As just shown, we have
$${\overline H}{}_f^h=\frac{1}{h}\int f(Y)K(.,Y)d\mu_h(Y)$$
and the same with $f$ replaced by $-f$, whence
${\overline H}{}_f^h={\underline H}{}_f^h$,
finite valued because $f$ is bounded. Thus $f$ is $h$-resolutive.
Let $c\ge0$ be a constant such that $|f|\le c$.
Then ${\overline H}{}_f^h=c-{\underline H}{}_{c-f}^h$
which is finely $h$-harmonic because $c-f\ge0$ and so $H_{c-f}^h$ is
$h$-invariant by Proposition \ref{prop6.1} (c) and
hence finely $h$-harmonic, being finite valued.

Returning to a general lower bounded $\mu_h$-measurable function $f$,
suppose first that $f$ is $h$-quasi\-resolutive. Then, as shown in the first
paragraph of the proof,
$\frac1h\int f(Y)K(.,Y)d\mu_h(Y)={\overline H}_f^h$ is finite q.e.\ on $U$.
Conversely, if $\frac1h\int f(Y)K(.,Y)d\mu_h(Y)<+\infty$ q.e.,
that is $\overline H{}_f^h\not\equiv+\infty$,
then Corollary \ref{cor4.8} applies to the increasing
sequence of bounded $\mu_h$-measurable and hence $h$-resolutive functions
$f\wedge j$ converging to $f$, and we conclude that $f$ is $h$-resolutive
(in particular $h$-quasi\-resolutive) and hence that
$\frac1h\int f(Y)K(.,Y)d\mu_h(Y)={\overline H}_f^h$ is finite everywhere on $U$.
\end{proof}

\begin{cor}\label{cor4.10a} Let $f:\Delta(U)\longrightarrow\overline\RR$
be $\mu_h$-measurable. Then $f$ is $h$-resolutive
if and only if $|f|$ is $h$-resolutive.
\end{cor}

\begin{proof} If $f$ is $h$-resolutive, and therefore $h$-quasiresolutive by
Lemma \ref{lemma3.3}, then $|f|=f\vee(-f)$ is
$h$-quasi\-resolutive according to 3.\ and 1.\ in Proposition \ref{prop6.3}.
Since $|f|$ is lower bounded (and $\mu_h$-measurable) then by Proposition
\ref{prop6.5} $|f|$ is even $h$-resolutive and $|f|K(x,.)$ is
$\mu_h$-integrable for every $x\in U$. So are therefore $f^+K(x,.)$ and
$f^-K(x,.)$, and it follows, again by Proposition \ref{prop6.5},
that $f^+$ and $f^-$ are $h$-resolutive.
So is therefore $f=f^+-f^-$ by Lemma \ref{lemma6.3c}.
\end{proof}

\begin{prop}\label{prop6.10b} Every $h$-quasi\-resolutive
function $f:\Delta(U)\longrightarrow\overline\RR$ is $\mu_h$-measurable.
\end{prop}

\begin{proof}
We begin by proving this for $f=1_A$, the indicator function of a subset $A$ of
$\Delta(U)$, cf.\ \cite[p.\ 113]{Do}. Clearly, $\dot H{}_f^h$ and
${H}{\vrule width 0pt depth 0.2em }_{\kern-0.7em\hbox{.}}\,\,{}_f^h(x)$ have their
values in $[0,1]$, and hence $\overline H{}_f^h=\dot H{}_f^h$ and
$\underline H{}_f^h
={H}{\vrule width 0pt depth 0.2em }_{\kern-0.7em\hbox{.}}\,\,{}_f^h$
according to Proposition \ref{prop6.1} (c).
Since $\overline{\cal U}{}_f^h$ is lower directed there is a decreasing
sequence of functions $u_j\in\overline{\cal U}{}_f^h$ such that
$\overline H{}_f^h(x_0)=\inf_ju_j(x_0)$. Replacing $u_j$ by
$u_j\wedge1\in\overline{\cal U}{}_f^h$ we arrange that $u_j\le1$.
Denote by $g_j$ the function defined on ${\overline U}$ by
$$g_j(Y)=\underset{z\to Y,\,z\in U}{\liminf}u_j(z)$$ for any
$Y\in {\overline U}.$
Clearly, $g_j$ is  l.s.c.\ on ${\overline U}$ and $1_A\le g_j\le1$
on $\Delta(U)$. Write $f_2=\inf_jg_j$ (restricted to $\Delta(U)$).
Then $f_2$ is Borel measurable and
$1\ge f_2\ge f=1_A$, whence $\overline H{}_{f_2}^h\ge\overline H{}_f^h$. For the
opposite inequality note that $u_j\in\overline{\cal U}{}_{f_2}^h$ because
$g_j\ge f_2$. Hence
$\overline H{}_f^h=\inf_ju_j\ge\overline H{}_{f_2}^h$ with equality at $x_0$.
Furthermore, $\overline H{}_f^h$ is invariant according to
Proposition \ref{prop6.1} (c), and hence
$\overline H{}_{f_2}^h-\overline H{}_f^h$ is positive and finely
$h$-superharmonic on $U$. Being $0$ at $x_0$ it is identically $0$, and so
$\overline H{}_{f_2}^h=\overline H{}_f^h$. Similarly there is
a positive Borel measurable function $f_1\le f$ such that
$\underline H{}_{f_1}^h=\underline H{}_f^h$. Clearly, $0\le f_1\le f_2\le1$.
Since $f$ is $h$-quasi\-resolutive
we obtain from Proposition \ref{prop6.1} (a) q.e.\ on $U$
$$
H_f^h=\underline H{}_{f_1}^h\le\overline H{}_{f_1}^h\le H_f^h
\le\underline H{}_{f_2}^h\le\overline H{}_{f_2}^h=H_f^h,
$$
thus with equality q.e.\ all through. Hence $f_1$ and $f_2$ are
$h$-quasi\-resolutive, and so is therefore $f_2-f_1$ by 1.\ and 2.\ in
Proposition \ref{prop6.3}, which also shows that $H_{f_2-f_1}^h=H_{f_2}-H_{f_1}=0$
q.e. Because $f_2-f_1$ is positive and Borel measurable on $\Delta(U)$ it
follows by Proposition \ref{prop6.5} that
$\frac1h\int(f_2(Y)-f_1(Y))K(.,Y)d\mu_h(Y)=0$,
and hence $f_1=f_2$ $\mu_h$-a.e.
It follows that $f=1_A$ is $\mu_h$-measurable, and so is therefore $A$.

Next we treat the case of a finite valued $h$-quasi\-resolutive function $f$ on
$\Delta(U)$. Adapting the proof given in \cite[p.\ 115]{Do} in the classical
setting we consider the space $\cal C(\overline\RR,\RR)$ of continuous (hence
bounded) functions $\overline\RR\to\RR$, and denote by $\Phi$ the space of
functions $\phi\in\cal C(\overline\RR,\RR)$ such that $\phi\circ f$ is
$h$-quasi\-resolutive.
In view of Proposition \ref{prop6.3}, $\Phi$ is a vector lattice,
closed under uniform convergence because $|\phi_j-\phi|<\eps$ implies
$|H_{\phi_j\circ f}^h-\overline H{}_{\phi\circ f}^h|\le\eps$ and
$|H_{\phi_j\circ f}^h-\underline H_{\phi\circ f}^h|\le\eps$ on $U\setminus E^h_f$,
and so $|\overline H{}_{\phi\circ f}^h-\underline H{}_{\phi\circ f}^h|\le2\eps$
on $U\setminus E^h_f$. We infer that
$\overline H{}_{\phi\circ f}^h=\underline H{}_{\phi\circ f}^h$ (finite values) q.e.\
on $U$, and so $\phi\circ f$ is indeed resolutive.
Furthermore, $\Phi$ includes the fuctions $\phi_n:t\longmapsto(1-|t-n|)\vee0$ on
$\RR$ for integers $n\ge1$, again by Proposition \ref{prop6.3}. These functions
separate points of $\RR$. In fact, for distinct $s,t\in\RR$, say $s<t$, take
$n=[s]$ (that is, $n\le s<n+1$). If also $n\le t<n+1$ then clearly
$\phi_n(t)<\phi_n(s)\le1$, and in the remaining case $t\ge n+1$ we have
$\phi_n(t)=0<\phi_n(s)$. It therefore follows by the lattice version of the
Stone-Weierstrass theorem that $\Phi=\cal C(\overline\RR,\RR)$.
Next, the class $\Psi$ of (not necessarily continuous) functions
$\psi:\RR\longrightarrow\RR$ for which $\psi\circ f$ is $h$-quasi\-resolutive
is closed under bounded monotone convergence, by Corollary \ref{cor4.8}
(adapted to a bounded monotone convergent sequence of functions $f_j$).
Along with the continuous functions $\RR\longmapsto\RR)$, $\Psi$ therefore
includes every bounded Borel measurable function $\RR\longmapsto\RR$.
In particular, the indicator function $1_J$ of an interval $J\subset\RR$
belongs to $\Psi$, and hence $1_J\circ f$ is $h$-quasi\-resolutive.
We conclude by the first part of the proof that the $h$-quasi\-resolutive
indicator function $1_J\circ f=1_{f^{-1}(J)}$ is $\mu_h$-measurable.

Finally, for an arbitrary $h$-quasi\-resolutive function
$f:\Delta(U)\longrightarrow\overline\RR$, write $A_+:=\{f=+\infty\}$ and
$A_-:=\{f=-\infty\} $. By 3.\ in Proposition \ref{prop6.3}, $f\vee0$ is
$h$-quasi\-resolutive and
$$
\overline H{}_{(+\infty)1_{A_+}}^h=\overline H{}_{f1_{A_+}}^h
\le\overline H{}_{f\vee0}^h=\dot H{}_{f\vee0}^h<+\infty
$$
on $U\setminus(E^h_f\cup P_f^h)$, cf.\ the opening of the present section.
It follows that $\overline H{}_{1_{A_+}}=0$ q.e.\ on $U$
and hence $\overline H{}_{(+\infty)1_{A_+}}^h=0$ q.e.\ according to 3.\ in
Proposition \ref{prop6.2}. Since $-f$ likewise is $h$-quasi\-resolutive
we have $\underline H{}_{(+\infty)1_{A_-}}^h=0$ q.e. Furthermore,
$0\le\underline H{}_{(+\infty)1_{A_+}}\le\overline H{}_{(+\infty)1_{A_+}}=0$ q.e.\ by
Proposition \ref{prop6.1} (a), and similarly $\overline H{}_{(+\infty)1_{A_-}}=0$
q.e. Writing $A=A_+\cup A_-=\{|f|=+\infty\}$ we infer from 2.\ in
Proposition \ref{prop6.3} that $1_A=1_{A_+}+1_{A_-}$ is $h$-quasi\-resolutive,
and hence so is
$(+\infty)1_A$ in view of Corollary \ref{cor4.8}. As shown in the first
paragraph of the proof it follows that $A$ is $\mu_h$-measurable.
Define $g:\Delta(U)\longrightarrow\RR$ by $g=f$ except that $g=0$ on
$A=\{|f|=+\infty\}$. Then $g\le f+(+\infty)1_A$ and $f\le g+(+\infty)1_A$, and
hence by Proposition \ref{prop6.2} (recall that $\overline H{}_{1_{A_+}}=0$
q.e.\ on $U$)
$$
\overline H{}_g^h\le H_f^h+\overline H{}_{(+\infty)1_A}^h=H_f^h
\le\underline H{}_g^h+\underline H{}_{(+\infty)1_A}^h=\underline H{}_g^h
$$
quasieverywhere, actually with equalities q.e., and hence $g$ is
$h$-quasi\-resolutive along with $f$.
As established in the beginning of the proof, $g$ is $\mu_h$-measurable and so
is $f$. In fact, $\{f\ne g\}=A$ is $\mu_h$-measurable with $\mu_h(A)=0$ by the
following corollary.
\end{proof}

\begin{cor}\label{cor6.10c} {\rm{(Cf.\ \cite[p.\ 114]{Do}.)}}
A set $A\subset\Delta(U)$ is $h$-harmonic null
if and only if $A$ is $\mu_h$-measurable with $\mu_h(A)=0$.
\end{cor}

\begin{proof} If $A$ is $\mu_h$-measurable with $\mu_h(A)=0$ then
$\overline H{}_{1_A}^h=\frac1h\int1_AK(.,Y)d\mu_h(Y)=0$ according to Proposition
\ref{prop6.5}. Conversely, if $\overline H{}_{1_A}^h=0$ then
$0\le\underline H{}_{1_A}^h\le\overline H{}_{1_A}^h=0$ and so $1_A$ is
$h$-resolutive. It follows by Proposition \ref{prop6.10b} that $1_A$
is $\mu_h$-measurable, and again by Proposition \ref{prop6.5} that
$\frac1h\int1_AK(.,Y)d\mu_h(Y)=\overline H{}_{1_A}^h=0$. Since $K(x,.)>0$ is
l.s.c.\ by \cite[Proposition 3.2]{EF1} we conclude that $\mu_h(A)=0$.
\end{proof}

\begin{cor}\label{cor6.3b} A function $f$ on $\Delta(U)$ with values in
$\overline\RR$ is $h$-resolutive if and only if $f$ is $h$-quasi\-resolutive.
\end{cor}

\begin{proof} The `only if' part is contained in Lemma \ref{lemma3.3}.
For the `if' part, suppose that $f$ is $h$-quasi\-resolutive and hence
$\mu_h$-measurable, by Proposition \ref{prop6.10b}. If $f\ge0$ then
$f$ is $h$-resolutive according to Proposition \ref{prop6.5}.
For arbitrary $f:\Delta(U)\longrightarrow\overline\RR$ this applies to
$f^+$ and $f^-$, which are $h$-quasi\-resolutive
according to 3.\ in Proposition \ref{prop6.3}. Consequently, $f=f^+-f^-$
is likewise finely $h$-resolutive by 1.\ and 2.\ in Proposition \ref{prop6.3}.
\end{proof}

\begin{theorem}\label{thm6.8}
A function $f:\Delta(U)\longrightarrow\overline\RR$ is
$h$-resolutive if and only if the
function $Y\longmapsto f(Y)K(x,Y)$ on $\Delta(U)$ is $\mu_h$-integrable
for quasievery $x\in U$. In the affirmative  case $Y\longmapsto f(Y)K(x,Y)$
is $\mu_h$-integrable for every $x\in U$, and we have everywhere on $U$
$$
H{}_f^h=\frac{1}{h}\int f(Y)K(.,Y)d\mu_h(Y).
$$
\end{theorem}

\begin{proof} Suppose first that $f$ is $h$-resolutive.
By Proposition \ref{prop6.10b} $f$ is then $\mu_h$-measurable.
According to 3.\ in Proposition \ref{prop6.3} the function $|f|$ is also
$h$-resolutive, and it follows by Proposition \ref{prop6.5} that
$\overline H{}_{|f|}^h(x)=\frac1h\int |f(Y)|K(x,Y)d\mu_h(Y)<+\infty$
for every $x\in U$. Conversely, suppose that $fK(x,.)$ is $\mu_h$-integrable
for quasi\-every
$x\in U\setminus E$. For any $x\in U$, $K(x,.)>0$ is l.s.c.\ and hence
$\mu_h$-measurable on $\overline U$ by \cite[Proposition 3.2]{EF1}, and so
$f^+$ and $f^-$ must be $\mu_h$-measurable. By Proposition \ref{prop6.5},
$f^+$ and $f^-$ are therefore $h$-quasi\-resolutive, that is $h$-resolutive
by Corollary \ref{cor6.3b}, and so is therefore
$f=f^+-f^-$ by 1.\ and 2.\ in Proposition \ref{prop6.3}.
\end{proof}

\begin{remark}\label{remark6.12}
In the case where $U$ is Euclidean open it
follows by the Harnack convergence theorem for harmonic functions (not
extendable to finely harmonic functions) for any numerical function
$f$ on $\Delta(U)$ that $\overline H{}_f^h$ is
$h$-hyper\-harmonic on $U$ (in particular $>-\infty$) and hence equal to
$\dot H{}_f^h$ (except if $\overline H{}_f^h\equiv-\infty$). It follows
by Proposition \ref{prop6.1} (a) that
$\overline H{}_f^h=\dot H{}_f^h
\ge{H}{\vrule width 0pt depth 0.2em }_{\kern-0.7em\hbox{.}}\,\,{}_f^h
=\underline H{}_f^h$. If $f$ is resolutive then
by Definition \ref{def6.9} equality prevails on all of $U$.
Summing up, Theorem \ref{thm6.8} is a (proper) extension of
the corresponding classical result, cf.\ e.g.\ \cite[Theorem 1.VIII.8]{Do}.
\end{remark}

We close this section with a brief discussion of an alternative, but
compatible concept of $h$-resolutivity based on the minimal-fine
(mf) topology, cf.\ \cite[Definition 3.4]{EF2}. The mf-closure of $U$
is $U\cup\Delta_1(U)$, and the relevant boundary functions $f$ are therefore
now defined only on the mf-boundary $\Delta_1(U)$.

Given a function $f$ on $\Delta_1(U)$ with values in ${\overline \RR}$,
a finely $h$-hyper\-harmonic function $u$ on $U$ is now said
to belong to the upper PWB$^h$ class, denoted again by
${\overline {\cal U}}{}_f^h$, if $u$ is lower bounded and if
$$\underset{x\to Y,\,x\in U}{\mfliminf}\,u(x)\ge f(Y)
\quad\text{ for every }\;Y\in \Delta_1(U).$$
This leads to new, but similarly denoted concepts
 $$\dot H{}_f^h=\inf{\overline{\cal U}}{}_f^h,
\quad{\overline H}{}_f^h=\widehat{\dot H{}_f^h}
={\widehat\inf}\,{\overline {\cal U}}{}_f^h\;(\le\dot H{}_f^h),$$
and hence new concepts of $h$-quasi\-resolutivity and $h$-resolutivity.

When considering reduction $R_u^A$ and sweeping $\widehat R{}_u^A$ of a
finely $h$-hyperharmonic function $u$ on $U$ onto a set
$A\subset\overline U$ we similarly use the alternative, though actually
equivalent mf-versions
\cite[Definition 3.14]{EF2}, cf.\ \cite[Theorem 3.16]{EF2}.
This occurs in the proof of Proposition \ref{prop6.4} (after the first
display), where now $W\subset\overline U$ is mf-open (and contains $A$).

The changes as compared with the case of $h$-resolutivity relative to the
natural topology are chiefly as follows. A set $A\subset U\cup\Delta_1(U)$
is of course now said to be $h$-harmonic null if $\overline H{}_{1_A}^h=0$
(with the present mf-version of $\overline H$). In Proposition
\ref{prop6.1} (b) we apply \cite[Proposition 3.12]{EF2} in place of its
corollary. In the beginning of the proof of Proposition \ref{prop6.10b}
the function $g_j$ shall now be defined at $Y\in\Delta_1(U)$ by
$$
g_j(Y)=\underset{z\to Y,\,z\in U}{\mfliminf}\,u_j(z).
$$
And $g_j$ is $\mu_h$-measurable on $\Delta_1(U)$ because $g_j$ equals
$\mu_h$-a.e.\ the $\mu_h$-measurable function defined $\mu_h$-a.e.\ on
$\Delta_1(U)$ by
$Y\longmapsto\mflim_{z\to Y,\,z\in U}\,u_j(z)=\frac{d\mu_{u_j}}{d\mu_h}(Y)$
according to the version of the
Fatou-Na{\"i}m-Doob theorem established in \cite[Theorem 4.5]{EF1}. Here
$\mu_{u_j}$ denotes the representing measure for $u_j$, that is $K\mu_{u_j}=u_j$,
and $d\mu_{u_j}/d\mu_h$ denotes the Radon-Nikod\'ym derivative of the
$\mu_h$-continuous part of $\mu_{u_j}$ (carried by $\Delta_1(U)$) with respect
to $\mu_h$. The $\mu_h$-measurability of $g_j$ on $\Delta_1(U)$ thus established
is all that is needed for the proof of the mf-version of Proposition
\ref{prop6.10b}, replacing mostly $\Delta(U)$ with $\Delta_1(U)$.

The following result is established like Theorem \ref{thm6.8}

\begin{theorem}\label{thm6.8a} A function
$f:\Delta_1(U)\longrightarrow\overline\RR$ is $h$-resolutive relative to the
mf-topology if and only if the
function $Y\longmapsto f(Y)K(x,Y)$ on $\Delta_1(U)$ is $\mu_h$-integrable
for quasievery $x\in U$. In the affirmative  case $Y\longmapsto f(Y)K(x,Y)$
is $\mu_h$-integrable for every $x\in U$, and we have everywhere on $U$
$$
H{}_f^h=\frac{1}{h}\int f(Y)K(.,Y)d\mu_h(Y).
$$
\end{theorem}

\begin{cor}\label{cor6.8b} For any $h$-resolutive function
$f:\Delta(U)\longrightarrow\overline\RR$ relative to the natural topology,
the restriction of $f$ to $\Delta_1(U)$ is resolutive relative to the
mf-topology. Conversely, for any $h$-resolutive function
$f:\Delta_1(U)\longrightarrow\overline\RR$ relative to the mf-topology,
the extension of $f$ by $0$ on $\Delta(U)\setminus\Delta_1(U)$ is
$h$-resolutive relative to the natural topology.
\end{cor}

\section{Further equivalent concepts of $h$-resolutivity}\label{sec4}

We again consider functions $f:\Delta(U)\longrightarrow\overline{\RR}$.
We show that the equivalent concepts of $h$-resolutivity and
$h$-quasiresolutivity do not alter when $\overline H{}_f^h,\underline H{}_f^h$
in Definitions \ref{def6.9}, \ref{def2.3} are
replaced by $\dot H{}_f^h$ and
${H}{\vrule width 0pt depth 0.2em }_{\kern-0.7em\hbox{.}}\,\,{}_f^h$,
respectively. Recall from Proposition \ref{prop6.1} (c) that
$\dot H{}_f^h=\overline H{}_f^h$ if $\dot H{}_f^h<+\infty$. This applies,
in particular, to the indicator function $1_A$ of a set $A\subset\Delta(U)$.
Therefore the ``dot''-version of the concept of an $h$-harmonic null set
$A$ is identical with the version considered in Definition \ref{def4.1}.

\begin{definition}\label{def4.0} A function $f$ on $\Delta(U)$ with values in
${\overline \RR}$ is said to be $h$-dot-resolutive if $\dot H{}_f^h
={H}{\vrule width 0pt depth 0.2em }_{\kern-0.7em\hbox{.}}\,\,{}_f^h$ on $U$
and if this function, also denoted by $H_f^h$,
is neither identically $+\infty$ nor identically $-\infty$.
\end{definition}

For any function $f:\Delta(U)\longrightarrow\RR$ we consider the following
subset of $U$:
$$
\dot E{}_f^h=\{\dot H{}_f^h=-\infty\}\cup
\{{H}{\vrule width 0pt depth 0.2em }_{\kern-0.7em\hbox{.}}\,\,{}_f^h=+\infty\}
\cup\{\dot H{}_f^h\ne
{H}{\vrule width 0pt depth 0.2em }_{\kern-0.7em\hbox{.}}\,\,{}_f^h\}.
$$

\begin{lemma}\label{lemma4.3} A function $f$ on $\Delta(U)$ with values in
${\overline \RR}$ is $h$-quasi\-resolutive if and only if $f$ is
$h$-dot-quasi\-resolutive in the sense that $\dot E{}_f^h$ is polar, or
equi\-valently that the relations
$\dot H{}_f^h>-\infty$,
${H}{\vrule width 0pt depth 0.2em }_{\kern-0.7em\hbox{.}}\,\,{}_f^h<+\infty$, and
$\dot H{}_f^h
={H}{\vrule width 0pt depth 0.2em }_{\kern-0.7em\hbox{.}}\,\,{}_f^h$
all hold quasi\-everywhere on $U$.
\end{lemma}

\begin{proof} If these three relations hold q.e.\ on $U$ then analogously
$\overline H{}_f^h
\ge{H}{\vrule width 0pt depth 0.2em }_{\kern-0.7em\hbox{.}}\,\,{}_f^h
=\dot H{}_f^h>-\infty$ q.e.\ and similarly
$\underline H{}_f^h<+\infty$ q.e. But $\overline H{}_f^h=\dot H{}_f^h>-\infty$ q.e.\
on $\{\overline H{}_f^h>-\infty\}$, hence also q.e.\ on $U$. Similarly,
$\underline H{}_f^h<+\infty$ q.e.\ on $U$, and altogether
$\overline H{}_f^h=\underline H{}_f^h$ q.e.\ on $U$. Thus $f$ is
quasi\-resolutive. The converse is obvious.
\end{proof}

\begin{lemma}\label{lemma4.4} Every $h$-dot-resolutive function $f$ is
$h$-resolutive, and hence $h$-quasi\-resolutive (now also termed
$h$-dot-quasi\-resolutive).
\end{lemma}

\begin{proof} Suppose that $f$ is $h$-dot-resolutive then $f$. Then $f$ is
$h$-resolutive, for $\dot H{}_f^h,\overline H{}_f^h,\underline H{}_f^h$, and
${H}{\vrule width 0pt depth 0.2em }_{\kern-0.7em\hbox{.}}\,\,{}_f^h$
are all equal because there is equality in the general inequalities
$\dot H{}_f^h\ge\overline H{}_f^h
\ge{H}{\vrule width 0pt depth 0.2em }_{\kern-0.7em\hbox{.}}\,\,{}_f^h$ and
$\dot H{}_f^h\ge\underline H{}_f^h
\ge{H}{\vrule width 0pt depth 0.2em }_{\kern-0.7em\hbox{.}}\,\,{}_f^h$,
cf.\ Proposition \ref{prop6.1} (a). The rest follows from Lemma \ref{lemma3.3}.
\end{proof}

In view of Lemma \ref{lemma4.4}, an $h$-(dot-)quasi\-resolutive function
is $h$-dot-resolutive if and only if $\dot E_f^h=\varnothing$.
Assertions 1.\ and 2.\ of Proposition \ref{prop6.3} therefore remain valid
when $E$ is replaced throughout by $\dot E$.
The proof of the dot-version of eq.\ (3.1) uses 1.\ of Proposition
\ref{prop6.2} in place of 2.\ there.
The following lemma is analogous to Lemma \ref{lemma6.3c}:

\begin{lemma}\label{lemma4.5}  Let $f$ be an $h$-quasi\-resolutive function
on $\Delta(U)$. If $f^+$ and $f^-$ are $h$-dot-resolutive then so is $f$,
and the function
$H_f^h=\dot H{}_f^h
={H}{\vrule width 0pt depth 0.2em }_{\kern-0.7em\hbox{.}}\,\,{}_f^h$
on $U$ is finite valued.
\end{lemma}

\begin{proof} Since $f$ is $h$-quasi\-resolutive, so are $f^+,f^-$
(besides being $h$-dot-resolutive) by 3.\ in Proposition \ref{prop6.3}.
Hence the functions $H_{f^{\pm}}^h=\dot H{}_{f^{\pm}}^h
={H}{\vrule width 0pt depth 0.2em }_{\kern-0.7em\hbox{.}}\,\,{}_{f^{\pm}}^h$
are finite valued (co-polar subets of $U$ being non-void). From
$-f^-\le f\le f^+$ it therefore follows by Proposition \ref{prop6.1} (a) that
$$-\infty
<-{H}{\vrule width 0pt depth 0.2em }_{\kern-0.7em\hbox{.}}\,\,{}_{f^-}^h
\le{H}{\vrule width 0pt depth 0.2em }_{\kern-0.7em\hbox{.}}\,\,{}_f^h
\le\dot H{}_f^h\le\dot H{}_{f^+}^h<+\infty.$$
Applying 1.\ in Proposition \ref{prop6.2} to the sums
$f=f^++(-f)^+=f^+-f^-$ and $-f=f^--f_+$, which are both well defined on
$\Delta(U)$, we obtain
$$
\dot H{}_f^h\le H_{f^+}^h-H_{f^-}^h
\le{H}{\vrule width 0pt depth 0.2em }_{\kern-0.7em\hbox{.}}\,\,{}_f^h
$$
on all of $U$. It follows that $\dot H{}_f^h
={H}{\vrule width 0pt depth 0.2em }_{\kern-0.7em\hbox{.}}\,\,{}_f^h
=H_{f^+}^h-H_{f^-}^h$ holds there, again by Proposition \ref{prop6.1} (a).
\end{proof}

\begin{cor}\label{cor4.8a} Let $(f_j)$ be an increasing sequence of lower
bounded $h$-dot-resolutive functions
$\Delta(U)\longrightarrow\,]-\infty,+\infty]$, and let $f=\sup_jf_j$.
If $\dot H{}_f^h\not\equiv+\infty$ then $f$ is $h$-dot-resolutive.
 \end{cor}

\begin{proof} For every $j$ we have
${H}{\vrule width 0pt depth 0.2em }_{\kern-0.7em\hbox{.}}\,\,{}_f^h
\ge{H}{\vrule width 0pt depth 0.2em }_{\kern-0.7em\hbox{.}}\,\,{}_{f_j}^h
=\dot H{}_{f_j}^h$, and hence
${H}{\vrule width 0pt depth 0.2em }_{\kern-0.7em\hbox{.}}\,\,{}_f^h
\ge\sup_j\dot H{}_{f_j}^h=\dot H{}_f^h$
according to 4.\ in Proposition \ref{prop6.2}.
Here equality prevails on account of Proposition \ref{prop6.1} (a).
By hypothesis, $H_f^h\not\equiv+\infty$, and clearly $H_f^h>-\infty$, so we
conclude that $f$ indeed is $h$-dot-resolutive.
\end{proof}

\begin{prop}\label{prop6.4a} For any $\mu_h$-measurable subset $A$
of $\Delta(U)$ the indicator function $1_A$ is $h$-dot-resolutive and
{\rm{(3.3)}} holds.
In particular, the constant function $1$ on $\Delta(U)$ is
$h$-dot-resolutive and $H_1^h=1$.
\end{prop}

\begin{proof} Since $\dot H{}_{1_A}^h\le\dot H_1^h=1<+\infty$ it follows from
Proposition \ref{prop6.1} (c) that $\overline H{}_{1_A}^h=\dot H{}_{1_A}^h$,
and the assertions reduce to the analogous Proposition \ref{prop6.4}.
\end{proof}

\begin{prop}\label{prop6.5a} Let $f$ be a $\mu_h$-measurable
lower bounded function on $\Delta(U)$. Then
$$
{\dot H}{}_f^h=\frac{1}{h}\int f(Y)K(.,Y) d\mu_h(Y)>-\infty,$$
and ${\dot H}{}_f^h$ is either identically $+\infty$ or the sum of
an $h$-invariant function and a constant $\le0$.
Furthermore $f$ is $h$-quasi\-resolutive if and only if $f$ is
$h$-dot-resolutive,
and that holds if and only if
$\frac{1}{h}\int f(Y)K(.,Y)d\mu_h(Y)<+\infty$
q.e.\ on $U$, or equivalently: everywhere on $U$.
In particular, every bounded $\mu_h$-measurable function
$f:\Delta(U)\longrightarrow\RR$ is $h$-dot-resolutive.
\end{prop}

\begin{proof} In view of the case of $\dot H$ in 4.\ of Proposition
\ref{prop6.2} the proof of the analogous Proposition \ref{prop6.5} carries over
{\it{mutatis mutandis}}.
\end{proof}

\begin{cor}\label{cor6.10a} Let $f:\Delta(U)\longrightarrow\overline\RR$
be $\mu_h$-measurable. Then $f$ is $h$-dot-resolutive
if and only if $|f|$ is $h$-dot-resolutive.
\end{cor}

\begin{proof} If $f$ is $h$-dot-resolutive, and therefore
$h$-quasiresolutive by Lemma \ref{lemma4.4}, then $|f|=f\vee(-f)$ is
$h$-quasi\-resolutive according to 1.\ and 3.\ in Proposition \ref{prop6.3}.
Now, $|f|$ is lower bounded (and $\mu_h$-measurable), and $|f|$ is therefore
even $h$-dot-resolutive, by Proposition \ref{prop6.5a}.
Consequently, $|f|K(x,.)$ is $\mu_h$-integrable for every
$x\in U$. So are therefore $f^+K(x,.)$ and $f^-K(x,.)$.
Again by Proposition \ref{prop6.5a} it follows that $f^+$ and $f^-$ are
$h$-dot-resolutive along with $f^+$ and $f^-$ (positive) by Lemma
\ref{prop6.5a}. So is therefore $f=f^+-f^-$ by Lemma \ref{lemma4.5}.
\end{proof}

\begin{cor}\label{cor4.6} A function $f$ on $\Delta(U)$ with values in
$\overline\RR$ is $h$-dot-resolutive if and only if $f$ is
$h$-quasi\-resolutive.
\end{cor}

\begin{proof} The `only if' part is contained in Lemma \ref{lemma4.4}.
For the `if' part, suppose that $f$ is $h$-quasi\-resolutive and hence
$\mu_h$-measurable according to Proposition \ref{prop6.10b}. If $f\ge0$ then
$f$ is $h$-dot-resolutive according to Proposition \ref{prop6.5a}.
For arbitrary $f:\Delta(U)\longrightarrow\overline\RR$ this applies to
$f^+$ and $f^-$, which are $h$-quasi\-resolutive
according to 3.\ in Proposition \ref{prop6.3}. Consequently, $f=f^+-f^-$
is likewise finely $h$-dot-resolutive by 1.\ and 2.\ in Proposition
\ref{prop6.3}.
\end{proof}

\begin{theorem}\label{thm6.8c}
A function $f:\Delta(U)\longrightarrow\overline\RR$ is
$h$-dot-resolutive if and only if the
function $Y\longmapsto f(Y)K(x,Y)$ on $\Delta(U)$ is $\mu_h$-integrable  for
every $x\in U$, or equivaletly for quasievery $x\in U$.
In the affirmative case the solution of the PWB-problem
on $U$ with boundary function $f$ is
$$H{}_f^h:=\overline H{}_f^h=\underline H{}_f^h=\dot H{}_f^h
={H}{\vrule width 0pt depth 0.2em }_{\kern-0.7em\hbox{.}}\,\,{}_f^h
=\frac{1}{h}\int f(Y)K(.,Y)d\mu_h(Y).
$$
\end{theorem}

\begin{proof} The proof of the analogous Theorem \ref{thm6.8a} carries over
{\it{mutatis mutandis}}.
\end{proof}

The alternative concept of $h$-resolutivity relative to the mf-topology
discussed at the end of the preceding section likewise has a similarly
established compatible version based on $\dot H{}_f^h$ and
${H}{\vrule width 0pt depth 0.2em }_{\kern-0.7em\hbox{.}}\,\,{}_f^h$
instead of $\overline H{}_f^h$ and $\underline H{}_f^h$.

\end{document}